\documentclass[review,3p,square]{elsarticle}

\usepackage{amssymb}
\usepackage{amsmath}
\usepackage{amsthm}
\usepackage{stmaryrd}
\usepackage{bbm}
\usepackage{enumitem}

\usepackage{fancyhdr}
\usepackage{hyperref}
\hypersetup{%
  pdftitle={BSDEs},%
  pdfsubject={BSDEs},%
  pdfauthor={ShengJun FAN, Ying Hu},%
  pdfkeywords={BSDEs},%
  pdfstartview=FitH,%
  CJKbookmarks=true,%
  bookmarksnumbered=true,%
  bookmarksopen=true,%
  colorlinks=true, linkcolor=blue, urlcolor=blue, citecolor=blue, %
}

\usepackage{cleveref}
\crefname{thm}{Theorem}{Theorems}
\crefname{pro}{Proposition}{Propositions}
\crefname{lem}{Lemma}{Lemmas}
\crefname{rmk}{Remark}{Remarks}
\crefname{cor}{Corollary}{Corollaries}
\crefname{dfn}{Definition}{Definitions}
\crefname{ex}{Example}{Examples}
\crefname{section}{Section}{Sections}
\crefname{subsection}{Subsection}{Subsections}

\newcommand{\eps}{\varepsilon}

\newcommand{\To}{\rightarrow}
\newcommand{\as}{{\rm d}\mathbb{P}\times{\rm d} t-a.e.}

\newcommand{\ps}{\mathbb{P}-a.s.}

\newcommand{\F}{\mathcal{F}}
\newcommand{\E}{\mathbb{E}}

\newcommand{\s}{\mathcal{S}}

\newcommand{\mcal}{\mathcal{M}}

\newcommand{\T}{[0,T]}

\newcommand{\R}{{\mathbb R}}

\newcommand {\Dis}{\displaystyle}

\newtheorem{thm}{Theorem}[section]

\newtheorem{pro}[thm]{Proposition}
\newtheorem{rmk}[thm]{Remark}

\newtheorem{ex}[thm]{Example}

\journal{}
\begin{document}
\begin{frontmatter}

\title{{On the uniqueness of solutions to quadratic BSDEs with non-convex generators and unbounded terminal conditions}\tnoteref{found}}
\tnotetext[found]{Shengjun Fan is supported by the State Scholarship Fund from the China Scholarship Council (No. 201806425013). Ying Hu is partially supported by Lebesgue center of mathematics ``Investissements d'avenir" program-ANR-11-LABX-0020-01, by CAESARS-ANR-15-CE05-0024 and by MFG-ANR-16-CE40-0015-01. Shanjian Tang is partially supported by National Science Foundation of China (No. 11631004) and Science and Technology Commission of Shanghai Municipality (No. 14XD1400400).
\vspace{0.2cm}}


\author{Shengjun Fan\vspace{-0.5cm}\corref{cor1}}
\author{\ \ \ \ Ying Hu\corref{cor2}}
\author{\ \ \ \ Shanjian Tang$^\dag$\corref{cor3}}

\cortext[cor1]{\ School of Mathematics, China University of Mining and Technology, Xuzhou 221116, China. E-mail: f\_s\_j@126.com \vspace{0.2cm}}

\cortext[cor2]{\ Univ. Rennes, CNRS, IRMAR-UMR6625, F-35000, Rennes, France. E-mail: ying.hu@univ-rennes1.fr\vspace{0.2cm}}

\cortext[cor3]{${}^\dag$Department of Finance and Control Sciences, School of Mathematical Sciences, Fudan University, Shanghai 200433, China. E-mail: sjtang@fudan.edu.cn}

\begin{abstract}
We prove a uniqueness result of the unbounded solution for a quadratic backward stochastic differential equation whose terminal condition is unbounded and whose generator $g$ may be non-Lipschitz continuous in the state variable $y$, non-convex (non-concave) in the state variable $z$, and instead satisfies a strictly quadratic condition and an additional assumption. The key observation is that if the generator is strictly quadratic, then the quadratic variation of the first component of the solution admits an exponential moment. Typically, a Lipschitz perturbation of some convex (concave) function satisfies the additional assumption mentioned above. This generalizes some results obtained in \cite{BriandHu2006PTRF} and \cite{BriandHu2008PTRF}.
\vspace{0.2cm}
\end{abstract}

\begin{keyword}
Backward stochastic differential equation \sep Existence and uniqueness \sep \\ \hspace*{1.95cm}Quadratic growth \sep Unbounded terminal condition \sep Strictly quadratic condition. \vspace{0.2cm}

\MSC[2010] 60H10\vspace{0.2cm}
\end{keyword}

\end{frontmatter}
\vspace{-0.4cm}

\section{Introduction}
\label{sec:1-Introduction}
\setcounter{equation}{0}

Since the nonlinear backward stochastic differential equation (BSDE in short) was initially introduced in \cite{PardouxPeng1990SCL}, a lot of efforts have been made to study the well posedness, and many applications have been found in various fields such as mathematical finance, stochastic control and PDEs etc. In particular, quadratic BSDEs were first investigated in \cite{Kobylanski2000AP} for bounded terminal conditions, which have attracted much attention in recent years and it is the subject of this article.

We consider the following quadratic BSDE:
\begin{equation}\label{eq:1.1}
  Y_t=\xi+\int_t^T g(s,Y_s,Z_s){\rm d}s-\int_t^T Z_s \cdot {\rm d}B_s, \ \ t\in\T,
\end{equation}
where the terminal value $\xi$ is an unbounded random variable, and the generator $g$ has a quadratic growth in the variable $z$. In \cite{BriandHu2006PTRF}, the authors obtained the first existence result for this kind of BSDEs, when the terminal value has a certain exponential moment. The uniqueness results were established in \cite{BriandHu2008PTRF}, \cite{DelbaenHuRichou2011AIHPPS} and \cite{DelbaenHuRichou2015DCD} when the generator $g$ is Lipschitz continuous in $y$, and either convex or concave in $z$. The case of a non-convex generator $g$ was tackled in \cite{Richou2012SPA} and \cite{BriandAdrien2017Arxiv}, but more assumptions are imposed on the terminal value $\xi$ than the exponential integrability. In this paper, we prove a uniqueness result for the unbounded solution of quadratic BSDEs, where the generator $g$ may be non-Lipschitz and has a general growth in $y$, and non-convex (non-concave) in $z$, and no additional assumption is required on the terminal value. We suppose instead that the generator $g$ satisfies an additional assumption which holds typically for a (locally) Lipschitz perturbation of some convex (concave) function (see \ref{A:H4} and \cref{pro:2.3} for details), and is strictly quadratic, i.e., either
\[
g(\omega,t,y,z)\geq \frac{\bar\gamma}{2} |z|^2-\beta|y|-\alpha_t(\omega),\vspace{-0.1cm}
\]
or\vspace{-0.1cm}
\[
g(\omega,t,y,z)\leq -\frac{\bar\gamma}{2} |z|^2+\beta|y|+\alpha_t(\omega)\vspace{0.1cm}
\]
holds for two constants $\bar\gamma>0,\beta\geq 0$ and a nonnegative process $\alpha_\cdot$.
Under this condition, we can prove that if $(Y_\cdot, Z_\cdot)$ is a solution satisfying $\E[\sup_{t\in\T}\exp(p|Y_t|+p\int_0^t\alpha_s{\rm d}s)]<+\infty$ for some $p>0$, then there exists a constant $\eps>0$ such that $\E[\exp(\eps\int_0^T |Z_t|^2 {\rm d}t)]<+\infty$. See \cref{pro:2.2} for details.

Let us close this introduction by introducing some notations that will be used later. Fix the terminal time $T>0$ and a positive integer $d$, and let $x\cdot y$ denote the Euclidean inner product for $x,y\in \R^d$. Suppose that $(B_t)_{t\in\T}$ is a $d$-dimensional standard Brownian motion defined on some complete probability space $(\Omega, \F, \mathbb{P})$. Let $(\F_t)_{t\in\T}$ be the natural filtration generated by $B_\cdot$ and augmented by all $\mathbb{P}$-null sets of $\F$. All the processes are assumed to be $(\F_t)$-adapted.

Denote by ${\bf 1}_A(\cdot)$ the indicator of set $A$, and ${\rm sgn}(x):={\bf 1}_{x>1}-{\bf 1}_{x\leq 1}$. Let $a\wedge b$ be the minimum of $a$ and $b$, $a^-:=-(a\wedge 0)$ and $a^+:=(-a)^-$. For any real $p\geq 1$, let
$\s^p$ be the set of all progressively measurable and continuous real-valued processes $(Y_t)_{t\in\T}$ such that
$$\|Y\|_{{\s}^p}:=\left(\E[\sup_{t\in\T} |Y_t|^p]\right)^{1/p}<+\infty,\vspace{0.1cm}$$
and $\mcal^p$ the set of all progressively measurable $\R^d$-valued processes $(Z_t)_{t\in\T}$ such that
$$
\|Z\|_{\mcal^p}:=\left\{\E\left[\left(\int_0^T |Z_t|^2{\rm d}t\right)^{p/2}\right] \right\}^{1/p}<+\infty.\vspace{0.2cm}
$$

As mentioned before, we will study BSDEs of type \eqref{eq:1.1}. The terminal condition $\xi$ is real-valued and $\F_T$-measurable, and the process $g(\cdot, \cdot, y, z):\Omega\times\T\To \R $
is progressively measurable for each pair $(y,z)$ and continuous in $(y,z)$. By a solution to \eqref{eq:1.1}, we mean a pair of progressively measurable processes $(Y_t,Z_t)_{t\in\T}$, taking values in $\R\times\R^d$ such that $\ps$, the function $t\mapsto Y_t$ is continuous, $t\mapsto Z_t$ is square-integrable, $t\mapsto g(t,Y_t,Z_t)$ is integrable, and verifies \eqref{eq:1.1}. And, we usually denote by BSDE $(\xi,g)$ the BSDE whose terminal condition is $\xi$ and whose generator is $g$.

Finally, we recall that a process $(X_t)_{t\in\T}$ belongs to class (D) if the family of random variables $\{X_\tau : \tau $ is any stopping time taking values in $\T\}$ is uniformly integrable.

\section{Main result}
\label{sec:2-main result}
\setcounter{equation}{0}

We define the following function, for any non-negative integrable function $f(\cdot):\T\To [0,+\infty)$ and any constants $\kappa\geq 0$ and $\lambda>0$:
\begin{equation}\label{eq:2.1}
\psi(s,x;f_\cdot,\kappa,\lambda)=\exp\left(\lambda e^{\kappa s}x+\lambda\int_0^s f_r e^{\kappa r}{\rm d}r\right),\ \ (s,x)\in \T\times [0,+\infty).
\end{equation}
It is easy to verify that $\psi(\cdot,\cdot;f_\cdot,\kappa,\lambda)$ belongs to $\mathcal{C}^{1,2}(\T\times [0,+\infty))$,
\begin{equation}\label{eq:2.2}
-\psi_x(s,x;f_\cdot,\kappa,\lambda)(f_s+\kappa x)+\psi_s(s,x;f_\cdot,\kappa,\lambda)=0, \ \ \ \  (s,x)\in\T\times[0,+\infty)
\end{equation}
and
\begin{equation}\label{eq:2.3}
-\lambda \psi_x(s,x;f_\cdot,\kappa,\lambda)+\psi_{xx}(s,x;f_\cdot,\kappa,\lambda)\geq 0, \ \ \ \  (s,x)\in\T\times[0,+\infty),
\end{equation}
where and hereafter, $\psi_s(\cdot,\cdot; f_\cdot,\kappa,\lambda)$ denotes the first-order partial derivative with respect to time, and $\psi_x(\cdot,\cdot; f_\cdot,\kappa,\lambda)$ and $\psi_{xx}(\cdot,\cdot; f_\cdot,\kappa,\lambda)$ are the first-order and second-order partial derivatives with respect to space of the time-space function $\psi(\cdot,\cdot; f_\cdot,\kappa,\lambda)$.

In the whole paper, we always fix a progressively measurable non-negative process $(\alpha_t)_{t\in\T}$ and several real constants $\beta\geq 0$, $0<\bar\gamma\leq \gamma$, $k\geq 0$, $\bar k \geq 0$ and $\delta\in [0,1)$. Let us first introduce the following two assumptions on the generator $g$.
\begin{enumerate}
\renewcommand{\theenumi}{(H\arabic{enumi})}
\renewcommand{\labelenumi}{\theenumi}
\item\label{A:H1} $\as$, for each $(y,z)\in \R\times\R^d$, it holds that
    $$
    {\rm sgn}(y)g(\omega,t,y,z)\leq \alpha_t(\omega)+\beta|y|+\frac{\gamma}{2} |z|^2;
    $$
\item\label{A:H2} There exists a deterministic nondecreasing continuous function $\phi(\cdot):[0,+\infty)\To [0,+\infty)$ with $\phi(0)=0$ such that $\as$, for each $(y,z)\in \R\times\R^d$,
    $$
    |g(\omega,t,y,z)|\leq \alpha_t(\omega)+\phi(|y|)+\frac{\gamma}{2} |z|^2.
    $$
\end{enumerate}

The following proposition gives a slight generalization of the existence result of \cite{BriandHu2008PTRF} for quadratic BSDEs with unbounded terminal conditions.

\begin{pro}\label{pro:2.1}
Suppose that the function $\psi$ is defined in \eqref{eq:2.1} and that $\xi$ is a terminal condition and $g$ is a generator which is continuous in $(y,z)$ and satisfies assumptions \ref{A:H1} and \ref{A:H2}.

(i) Let $(Y_\cdot,Z_\cdot)$ be a solution to BSDE $(\xi,g)$ such that $(\psi^p(t,|Y_t|;\alpha_\cdot,\beta, \gamma))_{t\in\T}$ belongs to class (D) for some $p\geq 1$. Then, $\ps$, for each $t\in \T$,
\begin{equation}\label{eq:2.4}
p\gamma |Y_t|\leq \psi^p(t,|Y_t|;\alpha_\cdot,\beta, \gamma)+ \frac{1}{2} p(p-1)\gamma^2\E\left[\int_t^T |Z_s|^2{\rm d}s \bigg|\F_t\right]\leq \E[\psi^p(T,|\xi|;\alpha_\cdot,\beta, \gamma)\big| \F_t].
\end{equation}

(ii) If $\E[\psi^p(T,|\xi|;\alpha_\cdot,\beta, \gamma)]<+\infty$ for some $p\geq 1$, then BSDE $(\xi,g)$ admits a solution such that $(\psi^p(t,|Y_t|;\alpha_\cdot,\beta, \gamma))_{t\in\T}$ belongs to class (D). Moreover, if $p>1$, then there exists a constant $C>0$ depending only on $p$ such that
\begin{equation}\label{eq:2.5}
\E\left[\sup\limits_{t\in \T}\psi^p\left(t,|Y_t|;\alpha_\cdot,\beta, \gamma\right)\right]\leq C\E[\psi^p(T,|\xi|;\alpha_\cdot,\beta, \gamma)]
\end{equation}
and $Z_\cdot\in \mcal^2$. And, if $p>2$, then $Z_\cdot\in \mcal^p$.
\end{pro}

\begin{proof} (i) Let $L_\cdot$ denote the local time of $Y_\cdot$ at $0$.  It\^{o}-Tanaka's formula applied to $\psi(s, |Y_s|; \alpha_\cdot,\beta, p\gamma)$ gives, in view of assumption \ref{A:H1},
$$
\begin{array}{ll}
&\Dis {\rm d}\psi(s,|Y_s|;\alpha_\cdot,\beta,p\gamma)\vspace{0.1cm}\\
=&\Dis -\psi_x(s,|Y_s|;\alpha_\cdot,\beta,p\gamma)
{\rm sgn}(Y_s)g(s,Y_s,Z_s){\rm d}s+\psi_x(s,|Y_s|;\alpha_\cdot,\beta,p\gamma){\rm sgn}(Y_s)Z_s \cdot {\rm d}B_s\vspace{0.1cm}\\
&\Dis +\psi_x(s,|Y_s|;\alpha_\cdot,\beta,p\gamma){\rm d}L_s+{1\over 2}\psi_{xx}(s,|Y_s|;\alpha_\cdot,\beta,p\gamma)|Z_s|^2{\rm d}s+\psi_s(s,|Y_s|;\alpha_\cdot,\beta,p\gamma){\rm d}s
\vspace{0.1cm}\\
\geq & \Dis \left[-\psi_x(s,|Y_s|;\alpha_\cdot,\beta,p\gamma)(\alpha_s+\beta |Y_s|)+\psi_s(s,|Y_s|;\alpha_\cdot,\beta,p\gamma)\right]{\rm d}s\vspace{0.1cm}\\
& \Dis
+{1\over 2}\left[-\gamma\psi_x(s,|Y_s|;\alpha_\cdot,\beta,p\gamma)|Z_s|^2+\psi_{xx}(s,|Y_s|;\alpha_\cdot,\beta,p\gamma)|Z_s|^2\right]{\rm d}s\vspace{0.1cm}\\
& \Dis +\psi_x(s,|Y_s|;\alpha_\cdot,\beta,p\gamma){\rm sgn}(Y_s) Z_s \cdot {\rm d}B_s.
\end{array}\vspace{0.1cm}
$$
Then, by virtue of \eqref{eq:2.2} and \eqref{eq:2.3} together with the fact that
$\psi_x(\cdot,\cdot;f_\cdot,\kappa,\lambda)\geq \lambda$, we have
\begin{equation}\label{eq:2.6}
{\rm d}\psi(s,|Y_s|;\alpha_\cdot,\beta,p\gamma)\geq \frac{1}{2}p(p-1)\gamma^2 |Z_s|^2 {\rm d}s+\psi_x(s,|Y_s|;\alpha_\cdot,\beta,p\gamma){\rm sgn}(Y_s) Z_s \cdot {\rm d}B_s,\ \ s\in \T.\vspace{0.1cm}
\end{equation}
Let us denote, for each $t\in\T$ and each integer $m\geq 1$, the following stopping time
$$
\tau_m^t:=\inf\left\{s\in [t,T]: \int_t^s \left(\psi_x(r,|Y_r|;\alpha_\cdot,\beta,p\gamma)\right)^2|Z_r|^2{\rm d}r\geq m \right\}\wedge T
$$
with the convention $\inf\Phi=+\infty$. It follows from \eqref{eq:2.6} and the definition of $\tau_m^t$ that for each $m\geq 1$,
$$
\psi(t,|Y_t|;\alpha_\cdot,\beta,p\gamma)+ \frac{1}{2}p(p-1)\gamma^2\E\left[\left.\int_t^{\tau_m^t}|Z_s|^2 {\rm d}s\right|\F_t\right] \leq \E\left[\left. \psi(\tau_m^t,|Y_{\tau_m^t}|;\alpha_\cdot,\beta,p\gamma) \right|\F_t\right],\ \ \ t\in \T.\vspace{0.1cm}
$$
Thus, since $(\psi(s,|Y_s|;\alpha_\cdot,\beta, p\gamma))_{s\in\T}$ belongs to class (D), the desired inequality \eqref{eq:2.4} follows by letting $m\To \infty$ and using Fatou's lemma in the last inequality.\vspace{0.1cm}

(ii) Thanks to (i), proceeding as in the proof of Proposition 3 in \cite{BriandHu2008PTRF} with a localization argument, we conclude that if $\E[\psi^p(T,|\xi|;\alpha_\cdot,\beta, \gamma)]<+\infty$ for some $p\geq 1$, then BSDE $(\xi,g)$ has a solution such that $(\psi^p(t,|Y_t|;\alpha_\cdot,\beta, \gamma))_{t\in\T}$ belongs to class (D) and \eqref{eq:2.4} holds. Moreover, for $p>1$, it is clear from \eqref{eq:2.4} that $Z_\cdot\in \mcal^2$. Since \eqref{eq:2.4} holds for $p=1$, we apply Doob's maximal inequality to get \eqref{eq:2.5}. Finally, the conclusion that $Z_\cdot\in \mcal^p$ for $p>2$ has been given in Corollary 4 of \cite{BriandHu2008PTRF}.
\end{proof}

To obtain a stronger integrability with respect to the process $Z_\cdot$, we need the following assumption, called strictly quadratic condition:\vspace{-0.1cm}
\begin{enumerate}
\renewcommand{\theenumi}{(H3)}
\renewcommand{\labelenumi}{\theenumi}
\item\label{A:H3} $\as$, for each $(y,z)\in \R\times\R^d$, it holds that
  \begin{equation}\label{eq:2.7}
   g(\omega,t,y,z)\geq \frac{\bar\gamma}{2} |z|^2-\beta|y|-\alpha_t(\omega),\vspace{-0.1cm}
  \end{equation}
   or\vspace{-0.1cm}
  \begin{equation}\label{eq:2.8}
    g(\omega,t,y,z)\leq -\frac{\bar\gamma}{2} |z|^2+\beta|y|+\alpha_t(\omega).\vspace{0.1cm}
  \end{equation}
\end{enumerate}

\begin{pro}\label{pro:2.2}
Let $\psi$ be defined in \eqref{eq:2.1}, $\xi$ be a terminal condition, $g$ be a generator satisfing \ref{A:H3}, and $(Y_\cdot,Z_\cdot)$ be a solution to BSDE $(\xi,g)$. If $\E[\sup_{t\in\T}\psi(t,|Y_t|;\alpha_\cdot,0,p_0)]<+\infty$ for some $p_0>0$, then for each $\eps\in (0,\eps_0]$ with $\eps_0:={\bar\gamma^2\over 18}\wedge {p_0\bar\gamma\over 12+6\beta T}$, we have
\begin{equation}\label{eq:2.9}
\E\left[\exp\left(\eps \int_0^T |Z_s|^2{\rm d}s\right)\right]<+\infty.
\end{equation}
In particular, for each $p>0$ and $\delta\in [0,1)$, $\E[\exp(p\int_0^T |Z_s|^{1+\delta}{\rm d}s)]<+\infty$.
\end{pro}

\begin{proof}
We only consider the case that \eqref{eq:2.7} holds. The other case is similar.  Since $(Y_\cdot,Z_\cdot)$ is a solution to BSDE $(\xi,g)$ and \eqref{eq:2.7} holds, we have for each $n\geq 1$,
$$
\frac{\bar\gamma}{2}\int_0^{\sigma_n} |Z_s|^2{\rm d}s \leq \Dis  Y_0-Y_{\sigma_n}+\int_0^{\sigma_n}(\alpha_s+\beta |Y_s|){\rm d}s+\int_0^{\sigma_n} Z_s\cdot {\rm d}B_s\leq X+\int_0^{\sigma_n} Z_s\cdot {\rm d}B_s,
$$
where $X:=(2+\beta T)\sup_{t\in\T}|Y_t|+\int_0^T\alpha_s{\rm d}s$, and
$$
\sigma_n:=\inf\left\{s\in [0,T]: \int_0^s |Z_r|^2{\rm d}r\geq n \right\}\wedge T.
$$
Then, for each $\eps>0$ such that $3\eps(2+\beta T)\leq p_0$, we have
$$
\exp\left(\frac{\bar\gamma}{2}\eps\int_0^{\sigma_n} |Z_s|^2{\rm d}s\right) \leq \exp(\eps X)\exp\left(\eps\int_0^{\sigma_n} Z_s\cdot {\rm d}B_s-\frac{3}{2}\eps^2\int_0^{\sigma_n} |Z_s|^2{\rm d}s\right)\exp\left(\frac{3}{2}\eps^2\int_0^{\sigma_n} |Z_s|^2{\rm d}s\right).
$$
Observe that the process
$$
H(t):=\exp\left(3\eps\int_0^{t\wedge\sigma_n} Z_s\cdot {\rm d}B_s-\frac{9}{2}\eps^2\int_0^{t\wedge\sigma_n} |Z_s|^2{\rm d}s\right)
$$
is a positive martingale with $H(0)=1$. By taking mathematical expectation in the last inequality and applying H\"{o}lder's inequality, we obtain
$$
\E\left[\exp\left(\frac{\bar\gamma}{2}\eps\int_0^{\sigma_n} |Z_s|^2{\rm d}s\right)\right] \leq \left(\E\left[\exp(3\eps X)\right]\right)^{1\over 3} \left(\E\left[\exp\left(\frac{9}{2}\eps^2\int_0^{\sigma_n}|Z_s|^2{\rm d}s\right)\right]\right)^{1\over 3}.
$$
Consequently, for $\eps\leq \bar\gamma/9$, we have
$$
\left(\E\left[\exp\left(\frac{\bar\gamma}{2}\eps\int_0^{\sigma_n} |Z_s|^2{\rm d}s\right)\right]\right)^{2\over 3} \leq \left(\E\left[\exp(3\eps X)\right]\right)^{1\over 3}<+\infty,
$$
which yields the inequality \eqref{eq:2.9} immediately from Fatou's lemma. Finally, for each $p>0$, $\delta\in [0,1)$, $x\geq 0$ and $\eps\in (0,\eps_0]$, it follows from Young's inequality that
$$
px^{1+\delta}=\left({2\over 1+\delta}\eps x^2\right)^{1+\delta\over 2}\left(p^{2\over 1-\delta}\left({2\over 1+\delta}\eps\right)^{-\frac{1+\delta}{1-\delta}}\right)^{1-\delta\over 2} \leq \eps x^2+{1-\delta\over 2}p^{2\over 1-\delta}\left({1+\delta\over 2\eps}\right)^{\frac{1+\delta}{1-\delta}}.
$$
Thus, the desired conclusion follows from \eqref{eq:2.9}. The proof is complete.
\end{proof}

In what follows, the following assumption on the generator $g$ will be used.
\begin{enumerate}
\renewcommand{\theenumi}{(H4)}
\renewcommand{\labelenumi}{\theenumi}
\item\label{A:H4} $\as$, for each $(y_i,z_i)\in \R\times\R^d$, $i=1,2$ and each $\theta\in (0,1)$, it holds that
  \begin{equation}\label{eq:2.10}
  \begin{array}{ll}
   &\Dis {\bf 1}_{\{y_1-\theta y_2>0\}}\left(g(\omega,t,y_1,z_1)-\theta g(\omega,t,y_2,z_2)\right)\vspace{0.1cm}\\
   \leq & \Dis (1-\theta)\left(\beta \left|\delta_\theta y\right|+\gamma \left|\delta_\theta z\right|^2+h(\omega,t,y_1,y_2,z_1,z_2,\delta)\right)
   \end{array}
  \end{equation}
   or
  \begin{equation}\label{eq:2.11}
  \begin{array}{ll}
   &\Dis {\bf 1}_{\{\theta y_1-y_2>0\}}\left(\theta g(\omega,t,y_1,z_1)- g(\omega,t,y_2,z_2)\right)\vspace{0.1cm}\\
   \leq &\Dis (1-\theta)\left(\beta \left|\bar\delta_\theta y\right|+\gamma \left|\bar\delta_\theta z\right|^2+h(\omega,t,y_1,y_2,z_1,z_2,\delta)\right),
  \end{array}
  \end{equation}
  where
  $$
  \delta_\theta y:=\frac{y_1-\theta y_2}{1-\theta},\ \ \  \delta_\theta z:=\frac{z_1-\theta z_2}{1-\theta},\ \ \ \bar\delta_\theta y:=\frac{\theta y_1- y_2}{1-\theta},\ \ \  \bar\delta_\theta z:=\frac{\theta z_1-z_2}{1-\theta},
  $$
and
  $$
  h(\omega,t,y_1,y_2,z_1,z_2,\delta):=\alpha_t(\omega)+k(|y_1|+|y_2|)
  +\bar k\left(|z_1|^{1+\delta}+ |z_2|^{1+\delta}\right).
  $$
\end{enumerate}

One typical example of \ref{A:H4} is
$$g(\omega,t,y,z):=g_1(z)+g_2(z),$$
where $g_1:\R^d\To\R$ is convex or concave with quadratic growth, and $g_2:\R^d\To\R$ is Lipschitz continuous, i.e., $g$ is a Lipschitz perturbation of some convex (concave) function.
More generally, we have

\begin{pro}\label{pro:2.3}
Assumption \ref{A:H4} holds for the generator $g$ as soon as it is continuous in $(y,z)$ and satisfies \ref{A:H1} together with anyone of the following conditions:
\begin{enumerate}
\item [(i)] $\as$, $g(\omega,t,\cdot,\cdot)$ is convex or concave;

\item [(ii)] $g$ is Lipschitz in the variable $y$ and $\delta$-locally Lipschitz in the variable $z$, i.e., $\as$, for each $(y_i,z_i)\in \R\times\R^d$, $i=1,2$, we have
\begin{equation}\label{eq:2.12}
|g(\omega,t,y_1,z_1)-g(\omega,t,y_2,z_2)|\leq \beta |y_1-y_2|+\gamma (1+|z_1|^\delta+|z_2|^\delta)|z_1-z_2|;\vspace{-0.1cm}
\end{equation}

\item [(iii)] $\as$, for each $(y,z)\in \R\times\R^d$, $g(\omega,t,\cdot,z)$ is Lipschitz, and $g(\omega,t, y,\cdot)$ is convex or concave;

\item [(iv)] $\as$, for each $(y,z)\in \R\times\R^d$, $g(\omega,t,\cdot,z)$ is convex or concave, and $g(\omega,t, y,\cdot)$ is $\delta$-locally Lipschitz, i.e., \eqref{eq:2.12} holds with $y_1=y_2=y$.
\end{enumerate}
\end{pro}

Before giving the proof of this proposition, we first make the following important remark.

\begin{rmk}\label{rmk:2.4}
It is easy to verify that assumption \ref{A:H4} also holds for the linear combination of two or several generators that is continuous in $(y,z)$ and satisfies assumption \ref{A:H1} together with anyone of conditions in \cref{pro:2.3} (with the same convexity or concavity when available), but with different parameters. Hence, the generator $g$ satisfying assumption \ref{A:H4} is not necessarily convex (concave) or Lipschitz with the variables $y$ and $z$, and it may have a general growth with respect to the variable $y$.

\end{rmk}

\begin{proof} [Proof of \cref{pro:2.3}] Given $(y_i,z_i)\in \R\times\R^d$, $i=1,2$ and $\theta\in (0,1)$.

(i) Assume that $\as$, $g(\omega,t,\cdot,\cdot)$ is convex. In view of \ref{A:H1}, if $\delta_\theta y>0$, then
$$
\begin{array}{lll}
g(\omega,t,y_1,z_1)&=&\Dis g\left(\omega,t,\theta y_2+(1-\theta)\delta_\theta y,\theta z_2+(1-\theta)\delta_\theta z\right)\vspace{0.1cm}\\
&\leq & \Dis \theta g(\omega,t,y_2,z_2)+(1-\theta)g\left(\omega,t,\delta_\theta y,\delta_\theta z\right)\vspace{0.1cm}\\
&\leq & \Dis \theta g(\omega,t,y_2,z_2)+(1-\theta)\left(\alpha_t(\omega)+\beta |\delta_\theta y|+\frac{\gamma}{2}|\delta_\theta z|^2\right).
\end{array}
$$
Thus, the inequality \eqref{eq:2.10} holds with $\gamma/2$ instead of $\gamma$, $k=0$ and $\bar k=0$. The concave case is similar.

(ii) Let the inequality \eqref{eq:2.12} holds. Note by \ref{A:H1} that $|g(\omega,t,0,0)|\leq \alpha_t(\omega)$. Then, in view of the fact that $2\delta<1+\delta$ and by virtue of Young's inequality we can deduce that for each $\eps>0$,
$$
\begin{array}{ll}
& \Dis g(\omega,t,y_1,z_1)-\theta g(\omega,t,y_2,z_2)\vspace{0.1cm}\\
\leq & \Dis |g(\omega,t,y_1,z_1)- g(\omega,t,y_2,z_2)|+(1-\theta)|g(\omega,t,y_2,z_2)|\vspace{0.1cm}\\
\leq & \Dis \beta |y_1-y_2|+\gamma \left(1+|z_1|^\delta+|z_2|^\delta\right)|z_1-z_2|+(1-\theta)
\left(|g(\omega,t,0,0)|+\beta |y_2|+\gamma \left(1+|z_2|^\delta\right)|z_2|\right)
\vspace{0.1cm}\\
\leq & \Dis \beta (|y_1-\theta y_2|+(1-\theta)|y_2|)+\gamma \left(1+|z_1|^\delta+|z_2|^\delta\right)(|z_1-\theta z_2|+ (1-\theta) |z_2|)\vspace{0.1cm}\\
& \Dis +(1-\theta)\left(\alpha_t(\omega)+\beta |y_2|+\gamma \left(1+|z_2|^\delta\right)|z_2|\right)\vspace{0.1cm}\\
\leq & \Dis (1-\theta)\left[\beta |\delta_\theta y|+2\beta |y_2|+\alpha_t(\omega)+\eps |\delta_\theta z|^2+c\left(1+|z_1|^{1+\delta}+|z_2|^{1+\delta}\right)\right],\vspace{0.1cm}\\
\end{array}
$$
where $c$ is a constant depending only on $(\gamma, \delta, \eps)$. Thus, the inequality \eqref{eq:2.10} holds with $\eps$, $\alpha_\cdot+c$, $2\beta$ and $c$ instead of $\gamma$, $\alpha_\cdot$, $k$ and $\bar k$ respectively.

(iii) Assume that $\as$, for each $(y,z)\in \R\times\R^d$, $g(\omega,t,\cdot,z)$ is Lipschitz, and $g(\omega,t, y,\cdot)$ is convex. Then, noticing by \ref{A:H1} that $|g(\omega,t,0,z)|\leq \alpha_t+\gamma|z|^2 /2 $, we have
$$
\begin{array}{ll}
&\Dis g(\omega,t,y_1,z_1)-\theta g(\omega,t,y_2,z_2)\vspace{0.1cm}\\
\leq & \Dis |g(\omega,t,y_1,z_1)-g(\omega,t,y_2,z_1)|+g(\omega,t,y_2,z_1)-\theta g(\omega,t,y_2,z_2)\vspace{0.1cm}\\
\leq & \Dis \beta |y_1-y_2|+g(\omega,t,y_2,\theta z_2+(1-\theta)\delta_\theta z)-\theta g(\omega,t,y_2,z_2)
\vspace{0.1cm}\\
\leq & \Dis \beta (|y_1-\theta y_2|+(1-\theta)|y_2|)+(1-\theta)\left(|g(\omega,t,y_2,\delta_\theta z)-g(\omega,t,0,\delta_\theta z)|+|g(\omega,t,0,\delta_\theta z)|\right)\vspace{0.1cm}\\
\leq & \Dis (1-\theta)\left(\beta |\delta_\theta y|+2\beta |y_2|+\alpha_t(\omega)+\frac{\gamma}{2} |\delta_\theta z|^2\right),\vspace{0.1cm}\\
\end{array}
$$
Thus, \eqref{eq:2.10} holds with $\gamma/2$ instead of $\gamma$, $2\beta$ instead of $k$, and $\bar k=0$ . The concave case is similar.

(iv) We only consider the case that $\as$, for each $(y,z)\in \R\times\R^d$, $g(\omega,t,\cdot,z)$ is convex, and $g(\omega,t, y,\cdot)$ is $\delta$-locally Lipschitz. In view of \ref{A:H1} and the fact that $2\delta<1+\delta$, we can apply Young's inequality to get that if $\delta_\theta y>0$, then for each $\eps>0$,
$$
\begin{array}{ll}
& \Dis g(\omega,t,y_1,z_1)-\theta g(\omega,t,y_2,z_2)\vspace{0.1cm}\\
\leq & \Dis g(\omega,t,y_1,z_1)- \theta g(\omega,t,y_2,z_1)+\theta |g(\omega,t,y_2,z_1)-g(\omega,t,y_2,z_2)|\vspace{0.1cm}\\
\leq & \Dis g(\omega,t,\theta y_2+(1-\theta)\delta_\theta y, z_1)-\theta g(\omega,t,y_2,z_1)
+\theta \gamma \left(1+|z_1|^\delta+|z_2|^\delta\right)|z_1-z_2|
\vspace{0.1cm}\\
\leq & \Dis (1-\theta)\left(|g(\omega,t, \delta_\theta y, z_1)-g(\omega,t, \delta_\theta y,0)|+g(\omega,t, \delta_\theta y,0)\right)\vspace{0.1cm}\\
& \Dis +\gamma \left(1+|z_1|^\delta+|z_2|^\delta\right)\left(|z_1-\theta z_2|+(1-\theta)|z_2|\right)\vspace{0.1cm}\\
\leq & \Dis (1-\theta)\left[\alpha_t(\omega)+\beta |\delta_\theta y|+\eps |\delta_\theta z|^2+c\left(1+|z_1|^{1+\delta}+|z_2|^{1+\delta}\right)\right],\vspace{0.1cm}\\
\end{array}
$$
where $c$ is a constant depending only on $(\gamma,\delta,\eps)$. Thus, the inequality \eqref{eq:2.10} holds with $\eps$, $\alpha_\cdot+c$, $\beta$ and $c$ instead of $\gamma$, $\alpha_\cdot$, $k$ and $\bar k$ respectively. The proposition is then proved.
\end{proof}

\begin{rmk}\label{rmk:2.5}
(i) Letting $y_1=y_2=y$ and $z_1=z_2=z$ in \eqref{eq:2.10} and \eqref{eq:2.11} respectively yields that
$$
{\bf 1}_{\{y>0\}} g(\omega,t,y,z)\leq \beta |y|+\gamma |z|^2 + \alpha_t(\omega)+2k |y|+2\bar k |z|^{1+\delta}\vspace{-0.1cm}
$$
and\vspace{-0.1cm}
$$
-{\bf 1}_{\{y<0\}} g(\omega,t,y,z)\leq \beta |y|+\gamma |z|^2 + \alpha_t(\omega)+2k |y|+2\bar k |z|^{1+\delta},
$$
whose combination implies assumption \ref{A:H1}.

(ii) Letting first $z_1=z_2=z$ in \eqref{eq:2.10} and \eqref{eq:2.11} and then letting $\theta\To 1$ yields that
$$
{\bf 1}_{\{y_1-y_2>0\}} \left(g(\omega,t,y_1,z)-g(\omega,t,y_2,z)\right)\leq \beta |y_1-y_2|,
$$
which means that $g$ satisfies the monotonicity condition in the state variable $y$.
\end{rmk}

The main result of this paper is stated as follows.\vspace{-0.1cm}

\begin{thm}\label{thm:2.6}
Suppose that the function $\psi$ is defined in \eqref{eq:2.1} and that $\xi$ is a terminal condition, $g$ is a generator which is continuous in the state variables $(y,z)$ and satisfies assumptions \ref{A:H1} and \ref{A:H2}, and $\E[\psi^p(T,|\xi|;\alpha_\cdot,\beta, \gamma)]<+\infty$ for each $p\geq 1$. Then, we have\vspace{0.1cm}

(i) If $g$ also satisfies assumption \ref{A:H4} with $\bar k=0$, then BSDE $(\xi,g)$ admits a unique solution $(Y_\cdot,Z_\cdot)$ such that for each $p\geq 1$, $\E\left[\sup_{t\in \T}\psi^p\left(t,|Y_t|;\alpha_\cdot,\beta, \gamma\right)\right]<+\infty$. Moreover, $Z_\cdot\in \mcal^p$ for each $p\geq 1$.\vspace{0.2cm}

(ii) If $g$ also satisfies assumptions \ref{A:H3} and \ref{A:H4}, then BSDE $(\xi,g)$ admits a unique solution $(Y_\cdot,Z_\cdot)$ such that for each $p\geq 1$, $\E\left[\sup_{t\in \T}\psi^p\left(t,|Y_t|;\alpha_\cdot,\beta, \gamma\right)\right]<+\infty$. Moreover, $\E[\exp(\eps \int_0^T |Z_s|^2{\rm d}s)]<+\infty$ for some $\eps>0$.
\end{thm}

\begin{proof}
The existence is a direct consequence of Propositions \ref{pro:2.1} and \ref{pro:2.2}. We now show the uniqueness part. Let us assume that \eqref{eq:2.10} in \ref{A:H4} holds.

Let both $(Y_\cdot,Z_\cdot)$ and $(Y'_\cdot, Z'_\cdot)$ be solutions to BSDE $(\xi,g)$ such that for each $p\geq 1$,
\begin{equation}\label{eq:2.13}
\E\left[\sup_{t\in \T}\psi^p\left(t,|Y_t|;\alpha_\cdot,\beta, \gamma\right)\right]<+\infty\ \  {\rm and}\ \ \E\left[\sup_{t\in \T}\psi^p\left(t,|Y'_t|;\alpha_\cdot,\beta, \gamma\right)\right]<+\infty.
\end{equation}
We use the $\theta$-difference technique developed in \cite{BriandHu2008PTRF}. For each fixed $\theta\in (0,1)$, define
$$
\delta_\theta  U_\cdot:=\frac{Y_\cdot-\theta Y'_\cdot}{1-\theta}\ \  {\rm and} \ \ \delta_\theta  V_\cdot:=\frac{Z_\cdot-\theta Z'_\cdot}{1-\theta}.
$$
Then, the pair $(\delta_\theta  U_\cdot,\delta_\theta  V_\cdot)$ solves the following BSDE:\vspace{0.1cm}
\begin{equation}\label{eq:2.14}
  \delta_\theta  U_t=\xi +\int_t^T \delta_\theta g(s) {\rm d}s-\int_t^T \delta_\theta  V_s \cdot {\rm d}B_s, \ \ \ \ t\in\T,\vspace{0.1cm}
\end{equation}
where
$$
\begin{array}{lll}
\Dis \delta_\theta g(s):= \Dis \frac{1}{1-\theta}\left[g(s,Y_s,Z_s)-\theta g(s,Y'_s,Z'_s)\right].
\end{array}
$$
It follows from \eqref{eq:2.10} that \vspace{-0.2cm}
\begin{equation}\label{eq:2.15}
{\bf 1}_{\{\delta_\theta U_s>0\}} \delta_\theta g(s)\leq \bar\alpha_s+\beta |\delta_\theta U_s|+\gamma |\delta_\theta V_s|^2,\vspace{-0.1cm}
\end{equation}
with\vspace{-0.1cm}
$$
\bar\alpha_s:=\alpha_s+k(|Y_s|+|Y'_s|)
+\bar k\left(|Z_s|^{1+\delta}+|Z'_s|^{1+\delta}\right).\vspace{0.1cm}
$$

(i) Let $\bar k=0$. In view of \eqref{eq:2.13}, from H\"{o}lder's inequality it is not hard to verify that
\begin{equation}\label{eq:2.16}
\E\left[\sup_{t\in \T}\psi \left(t,|\delta_\theta U_t|;\bar\alpha_\cdot,\beta, \gamma\right)\right]<+\infty.
\end{equation}
Thus, in view of \eqref{eq:2.14}, \eqref{eq:2.15} and \eqref{eq:2.16}, we apply It\^{o}-Tanaka's formula to $\psi(s,\delta_\theta U^+_s; \bar\alpha_\cdot,\beta, \gamma)$ and argue as in the proof of \cref{pro:2.1} to deduce that for each $t\in\T$,
$$
\gamma\delta_\theta U^+_t\leq \psi(t,\delta_\theta U^+_t; \bar\alpha_\cdot,\beta, \gamma)\leq \E\left[\psi(T,\xi^+; \bar\alpha_\cdot,\beta, \gamma) |\F_t\right] \leq \E\left[\psi(T,|\xi|; \bar\alpha_\cdot,\beta, \gamma) |\F_t\right],
$$
and then\vspace{-0.1cm}
$$
\gamma (Y_t-\theta Y'_t)^+\leq (1-\theta)\E\left[\psi(T,|\xi|; \bar\alpha_\cdot,\beta, \gamma) |\F_t\right].
$$
Letting $\theta\To 1$ in the last inequality yields that $\ps$, for each $t\in \T$, $Y_t\leq Y'_t$. Thus, the desired conclusion follows by interchanging the position of $Y_\cdot$ and $Y'_\cdot$.

(ii) Let assumption \ref{A:H3} holds. Thanks to \cref{pro:2.2}, we have $\E[\exp(p\int_0^T |Z_s|^{1+\delta}{\rm d}s)]<+\infty$ and $\E[\exp(p\int_0^T |Z'_s|^{1+\delta}{\rm d}s)]<+\infty$ for each $p\geq 1$. Then, in view of \eqref{eq:2.13}, from H\"{o}lder's inequality we can conclude that \eqref{eq:2.16} still holds. Thus, the same computation as above yields the uniqueness result.

Finally, another case that \eqref{eq:2.11} holds can be proved in the same way. The proof is then complete.
\end{proof}

\begin{rmk}\label{rmk:2.7}
In view of \cref{pro:2.3} and \cref{rmk:2.4}, it is clear that \cref{thm:2.6} generalizes the uniqueness result for quadratic BSDEs with unbounded terminal conditions obtained in \cite{BriandHu2008PTRF}.
\end{rmk}

\begin{ex}\label{ex:2.8}
For each $(\omega,t,y,z)\in \Omega\times \T\times\R \times \R^d$, define
$$
g_1(\omega,t,y,z)=|z|^2-|z|^{3\over 2}+\sin |z|+y^2{\bf 1}_{y\leq 0}-|y|+|B_t(\omega)|\vspace{-0.1cm}
$$
and\vspace{-0.1cm}
$$
g_2(\omega,t,y,z)=-|z|^2+\sin |z|^{4\over 3}+|z|-y^3 {\bf 1}_{y\geq 0}+\sin |y|+|B_t(\omega)|.
$$
By virtue of \cref{pro:2.3}, it is not hard to verify that both $g_1$ and $g_2$ are continuous in $(y,z)$ and satisfy assumptions \ref{A:H1}---\ref{A:H4}. However, they are non-convex (non-concave) with respect to the variable $z$, and non-Lipschitz with respect to the variable $y$.
\end{ex}

\vspace{0.3cm}



\setlength{\bibsep}{2pt}
\bibliographystyle{elsarticle-harv}

\end{document}